\documentclass[12pt]{amsart}
\usepackage{amsfonts, amsmath, amsthm, amssymb}

\usepackage{graphicx}
\usepackage{epstopdf}
\usepackage{psfrag}

\newtheorem{pro}{Proposition}[section]
\newtheorem{thm}[pro]{Theorem}
\newtheorem{lem}[pro]{Lemma}

\theoremstyle{definition}
\newtheorem{dfn}[pro]{Definition}

\newcommand{\VV}{\mathcal V}
\newcommand{\WW}{\mathcal W}

\newcommand{\bdy}{\partial}

\newcommand{\EE}{\mathcal E}

\newcommand{\TT}{\mathcal T}

\newcommand{\MV}{\mathbb V}
\newcommand{\MW}{\mathbb W}

\newcommand{\plex}[1]{\ensuremath{[{#1}]}}

\title{Normalizing Topologically Minimal Surfaces II: Disks}
\date{\today}
\author{David Bachman}

\begin{document}

\begin{abstract}
We show that a topologically minimal disk in a tetrahedron with index $n$ is either a normal triangle, a normal quadrilateral, or a normal helicoid with boundary length $4(n+1)$. This mirrors geometric results of Colding and Minicozzi. 
\end{abstract}
\maketitle

\section{Introduction}

\markright{NORMALIZING TOPOLOGICALLY MINIMAL SURFACES II}

A surface is said to be {\it geometrically minimal} if it represents a critical point for the area functional. In \cite{cm1}, \cite{cm2}, \cite{cm3}, \cite{cm4} Colding and Minicozzi showed that if the intersection of a geometrically minimal surface with a Euclidean 3-ball is simply connected, then it must be either the graph of a function or a helicoid, as in Figure \ref{f:helicoid}(a). Furthermore, geometrically minimal surfaces have a well-defined {\it index}  (defined as the Morse index of the critical point that they represent) and Colding and Minicozzi showed that the number of turns of such helicoids are directly proportional to their index. 

In earlier work, the author introduced the idea of a {\it topologically minimal surface}, and conjectured that such surfaces have many of the same properties as geometrically minimal surfaces \cite{TopIndexI}. This lead to the conjecture that such surfaces could be always be isotoped so that they satisfy a similar local classification as that given by Colding and Minicozzi. 

\begin{figure}
\[\begin{array}{cc}
\includegraphics[width=2in]{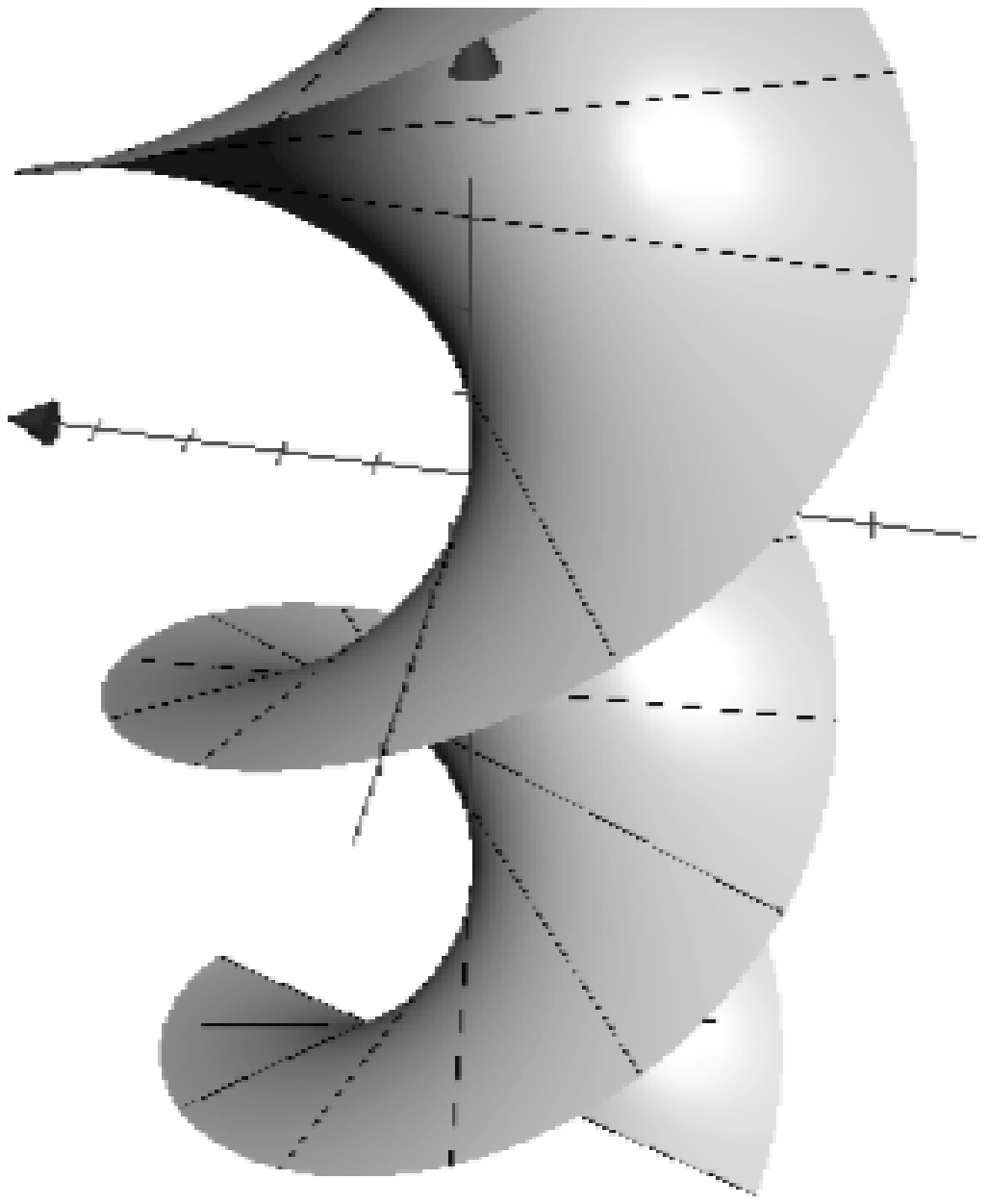} & \includegraphics[width=2in]{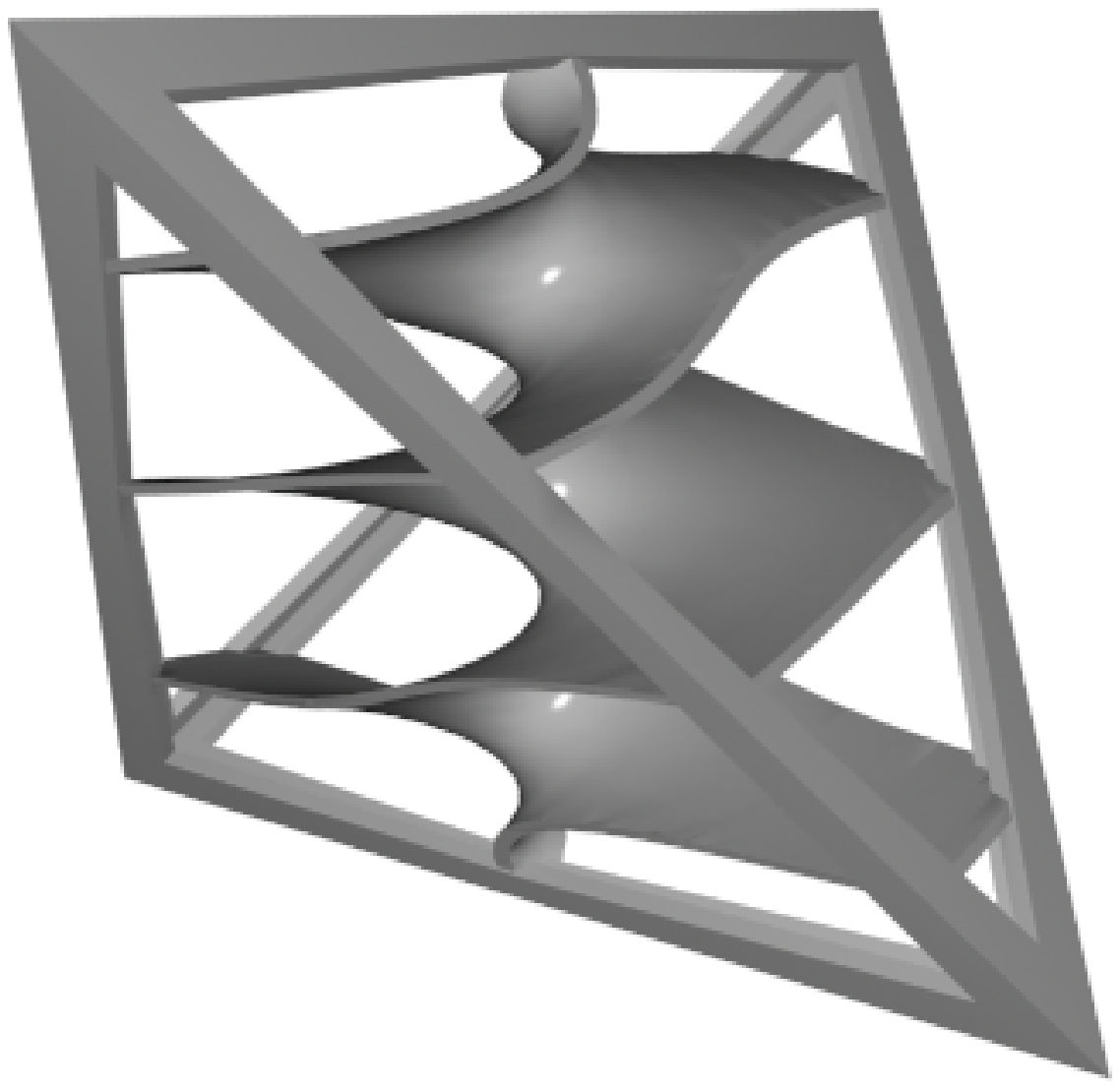}\\
\mbox{(a)} & \mbox{(b)}
\end{array}\]
\caption{(a) A helicoid, the graph of $z=\theta$ in cylindrical coordinates. (b) The intersection of a helicoid with a tetrahedron.}
\label{f:helicoid}
\end{figure}

In the first of the present series of papers \cite{TopMinNormalI}, the author introduced a separate notion of what it means for a surface {\it in a tetrahedron} to be topologically minimal, along with an associated notion of ``index". (See Section \ref{s:TopMinDefinitions} for definitions.) The main result of that paper is that if an embedded surface $H$ in a triangulated 3-manifold is topologically minimal with index $n$, then it could be isotoped so that its intersection with each tetrahedron is topologically minimal and the sum of the indices over all tetrahedra is at most $n$. 

In the present paper we begin the classification of all possible topologically minimal surfaces in tetrahedra by examining the simply connected case. In particular, our main result is the following:

\begin{thm}
\label{t:main}
Suppose $H$ is a topologically minimal disk in a tetrahedron. Then either the index of $H$ is
	\begin{enumerate}
		\item 0, and $H$ is a triangle or quadrilateral, or
		\item $n \ge 1$, and $H$ is a normal helicoid with $|\bdy H|=4(n+1)$.
	\end{enumerate}
\end{thm}

A {\it normal helicoid} is the intersection of the graph of $z=\theta$ (in cylindrical coordinates) with a tetrahedron, where the $z$-axis meets the midpoints of opposite edges, as in Figure \ref{f:helicoid}(b). Note that Theorem \ref{t:main}, combined with the aforementioned results from the first paper in this series \cite{TopMinNormalI},  precisely mirrors the Colding-Minicozzi result described above.

\section{Topologically Minimal Surfaces in Tetrahedra}
\label{s:TopMinDefinitions}

Throughout this paper $H$ will represent a properly embedded disk in a tetrahedron $\Delta$. The 1- and 2-skeleton of $\Delta$ will be denoted $\TT^1$ and $\TT^2$, respectively. 

\begin{dfn}
An edge-compressing disk for $H$ is a disk $E$ such that $\bdy E=\alpha \cup \beta$, where $E \cap H=\alpha$ and $E \cap \TT^1 =\beta$ is a subarc of an edge of $\TT^1$. We will use the notation $\EE(H)$ to refer to the set of edge-compressing disks for $H$. 
\end{dfn}

We define an equivalence relation on edge-compressing disks as follows:

\begin{dfn}
Edge-compressing disks $E,E' \in \EE(H)$ are {\it equivalent}, $E \sim E'$, if $E$ and $E'$ are isotopic through disks in $\EE(H)$. 
\end{dfn}

\begin{dfn}
Two edge-compressing disks $E$ and $E'$ are said to be {\it $\TT^1$-disjoint} if they are disjoint away from a neighborhood of $\TT^1$. 
\end{dfn}

\begin{dfn}
\label{d:DiskComplex}
The {\it edge-compresssing disk complex} $\plex{\EE(H)}$ is the complex defined as follows: vertices correspond to equivalence classes $[E]$ of the disk set $\EE(H)$. A collection of $n$ vertices spans an $(n-1)$-simplex if there are representatives of each which are pairwise $\TT^1$-disjoint.
\end{dfn}

\begin{dfn}
\label{d:Indexn}
The disk $H$ is said to be {\it topologically minimal in the tetrahedron $\Delta$} if the complex $\plex{\EE(H)}$ is either empty or non-contractible. In the former case we define the {\it index} of $\plex{\EE(H)}$ to be 0. In the latter case, the {\it index} of $\plex{\EE(H)}$ is defined to be the smallest $n$ such that  $\pi_{n-1}(\plex{\EE(H)})$ is non-trivial. 
\end{dfn}

\section{Disks with normal boundary.}

\begin{dfn}
A properly embedded arc $\delta$ in a face of $\TT^2$ is {\it normal} if it connects distinct edges of $\TT^1$. If $\bdy H$ is made up of normal arcs, then we say its {\it length} is $|H \cap \TT^1|$. A {\it normal triangle} is a properly embedded disk in $\Delta$ whose boundary has length 3, and a {\it normal  quadrilateral} is a disk whose boundary has length 4. 
\end{dfn}

Normal triangles and quadrilaterals were first introduced by Kneser \cite{kneser:29}, and later used extensively by Haken \cite{haken:68} to solve many important problems in the theory of 3-manifolds. Later, Rubinstein introduced {\it normal octagons} as part of his theory of ``almost normal surfaces"  \cite{rubinstein:93}.  In the parlance of the present paper, a normal octagon is a normal helicoid with boundary length 8. Many of Rubinstein's ideas about almost normal surfaces came from thinking about topological analogues of geometrically minimal surfaces. In that sense, the present series of papers is a continuation of his program.

\begin{lem}
If $H$ is a topologically minimal disk in $\Delta$, then $\bdy H$ is made up of normal arcs.
\end{lem}

\begin{proof}
If $\bdy H$ meets some face $\sigma$ of $\Delta$ in a non-normal arc, then an outermost such arc on $\sigma$ cuts off an edge-compressing disk $E$ for $H$. As $E \subset \bdy \Delta$, it is disjoint from all other edge-compressing disks. Hence, the vertex $[E]$ of $\plex{\EE(H)}$ is connected to every other vertex by an edge. It follows that $\plex{\EE(H)}$ is a cone on $[E]$, and is thus contractible. We conclude $H$ was not topologically minimal in $\Delta$.
\end{proof}

The combinatorics of normal loops, i.e. loops made up of normal arcs, in the boundary of a tetrahedron are well understood. Picturing a normal loop $c$ as the equator of the sphere $\bdy \Delta$, the following lemma and figure articulate the possibilities for each hemisphere of $\partial \Delta$ cut along $c$:

\begin{lem}
\label{l:tetCurve}
Let $c \subset \Delta$ be a normal loop pictured as the equator of $\bdy \Delta$ and  let $S_1$ and $S_2$ denote the hemispheres of $\Delta - c$. Then exactly one of the following holds:
	\begin{enumerate}
		\item $c$ has length 3 and one of $S_1$ or $S_2$ meets $\bdy \Delta$ as in Figure \ref{f:tetCurve}(a).   In particular $c$ is the link of a vertex of $\Delta$.
		\item $c$ has length 4 and both $S_1$ and $S_2$ meet $\bdy \Delta$ as in Figure \ref{f:tetCurve}(b).  In particular, $c$ separates a pair of edges.
		\item $c$ has length $4k, k \geq 2$, and both $S_1$ and  $S_2$ meet $\bdy \Delta$ as in Figure \ref{f:tetCurve}(c).  In particular, $S_i$ contains 2 vertices, three sub-edges meeting each vertex, and $2k-3$ parallel sub-edges separating the 2 vertices.
	\end{enumerate}
\end{lem}

The above lemma is a standard result in normal surface theory, and is a straightforward exercise in considering the possibilities for $c$ on the boundary of a tetrahedron (see e.g.~\cite{thompson:94}).  

	\begin{figure}[h]
               \psfrag{3}{$3$}
               \psfrag{4}{$4$}
               \psfrag{k}{$4k, k \geq 2$}
               \psfrag{c}{$c$}
               \psfrag{a}{(a)}
               \psfrag{b}{(b)}
               \psfrag{d}{(c)}
               \psfrag{n}{\tiny{($2k-3$ sub-edges)}}
               \begin{center}
                       \includegraphics[width=4 in]{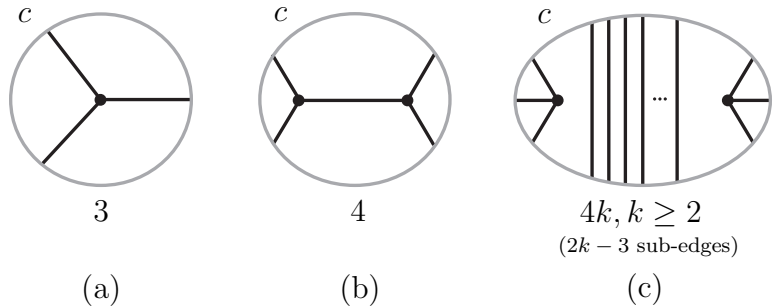}
                       \caption{Hemispheres of a tetrahedron  bounded by normal loops of lengths 3, 4, and $4k, k\geq 2$ }
                       \label{f:tetCurve}
               \end{center}
       \end{figure}

\begin{lem}
\label{l:NormalTriangle}
If $\plex{\EE(H)}=\emptyset$ then $H$ is a normal triangle or quadrilateral.
\end{lem}

\begin{proof}
Suppose $\bdy H$ meets some edge $e$ of $\TT^1$ twice. Then an adjacent pair of intersection points $a$ and $b$ cobound a subarc $\beta$ of $e$. Furthermore, $a$ and $b$ can be connected by an arc $\alpha$ on $H$. Since $H$ is a disk in a ball, the loop $\alpha \cup \beta$ bounds a disk $E$, which by definition is an edge-compressing disk for $H$. We conclude that when $\plex{\EE(H)}=\emptyset$ the surface $H$ meets each edge at most once. The only possibilities this leaves for $\bdy H$ are a loop of length 3 or a loop of length 4 (see Figure \ref{f:tetCurve}).
\end{proof}


Assuming the index of a topologically minimal disk $H$ in $\Delta$ is at least 1, the previous lemmas tell us that $\bdy H$ has length $4k$ (for some $k \geq 2$), and divides $\bdy \Delta$ into two hemispheres of the form depicted in Figure \ref{f:tetCurve}(c). Note that, while there is only one normal triangle and quadrilateral up to symmetries of the tetrahedron $\Delta$, there can be many disks with longer boundaries. For example, two different disks with boundary of length  20 are depicted in Figure \ref{f:20-gons}. However, for each possible boundary length, there is only one disk that can be realized as a normal helicoid.  One consequence of Theorem \ref{t:main} is that the other disks are not topologically minimal. This dispels a belief held by many experts that all disks with normal boundary should be topologically minimal. 

\begin{figure}
\[\includegraphics[width=4in]{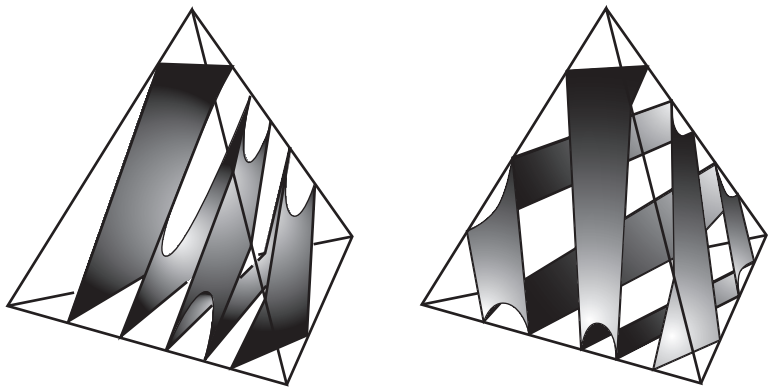}\]
\caption{Two different 20-gons in a tetrahedron $\Delta$. Only the one on the left can be realized as the intersection of a helicoid with $\Delta$.}
\label{f:20-gons}
\end{figure}

\section{Proof of Theorem \ref{t:main}}

Let $\VV$ and $\WW$ denote the two sides of $H$.  We now examine further the edge-compressing disks for $H$.  By definition, such a disk $E$ has boundary $\bdy E=\alpha \cup \beta$, where $\alpha \subset H$ and $\beta \subset \TT^1$. Since $H$ is a disk, the arc $\alpha$ is unique, up to isotopy fixing its endpoints. Picturing $H$ as a round, flat disk, the intersections of the edge-compressing disks in $\VV$ and $\WW$ with $H$ are isotopic to two sets of line segments, which we denote as $\MV$ and $\MW$, respectively. Furthermore, as the arcs of $\TT^1 \cap \VV$ are parallel on $\bdy \Delta -\bdy H$, it follows that the lines of $\MV$ are parallel. Similarly, the lines of $\MW$ must be parallel as well. See Figure \ref{f:12-gonCircle}.

\begin{figure}
\psfrag{1}{$1$}
\psfrag{2}{$2$}
\psfrag{3}{$3$}
\psfrag{4}{$4$}
\psfrag{5}{$5$}
\psfrag{6}{$6$}
\psfrag{7}{$7$}
\psfrag{8}{$8$}
\psfrag{9}{$9$}
\psfrag{a}{$10$}
\psfrag{b}{$11$}
\psfrag{c}{$12$}
\psfrag{x}{(a)}
\psfrag{y}{(b)}
\[\includegraphics[width=4.5in]{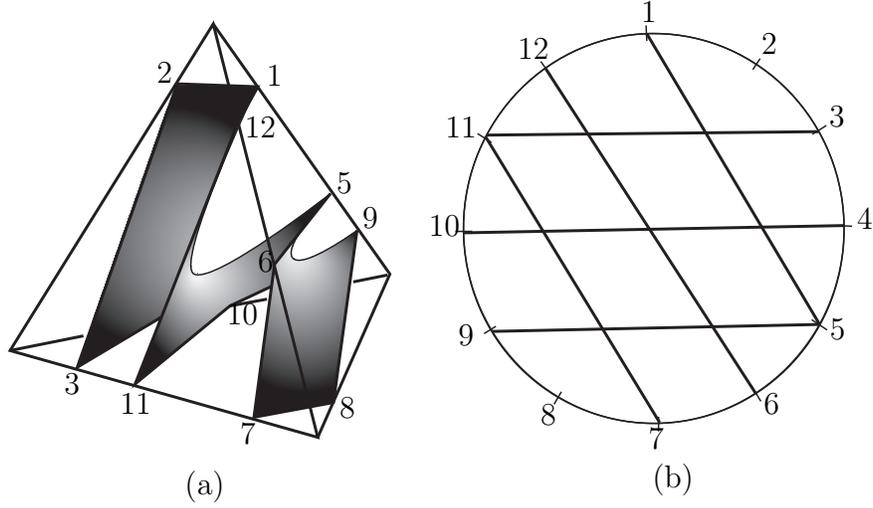}\]
\caption{(a) A 12-gon $H$ in $\Delta$, and (b) $H$ realized as a flat disk. The horizontal line segments $\MV$ in (b) are the intersections of the edge-compressing disks in $\VV$ with $H$, while the slanted line segments $\MW$ are from the edge-compressing disks in $\WW$.}
\label{f:12-gonCircle}
\end{figure}

Note that two edge-compressing disks intersect if and only if they intersect on $H$. Hence, the complex $\plex{\EE(H)}$ is completely determined by the lines in $\MV$ and $\MW$, and how these lines meet each other. That is, two edge compressing disks are $\TT^1$-disjoint, and hence represent endpoints of an edge of $\plex{\EE(H)}$, if and only if the corresponding line segments in $\MV \cup \MW$ meet in at most an endpoint.

Given that $H$ is topologically minimal in $\Delta$ with index $n$, our task is to determine 
	\begin{enumerate}
		\item how many line segments are in $\MV$ and $\MW$, and 
		\item how much one family is rotated with respect to the other. 
	\end{enumerate}

The next lemma will allow us to remove some of the disks in $\plex{\EE(H)}$ without changing it homotopy type. This, in turn, allows us to simplify the two families of line segments that we have just described. 

\begin{lem}
\label{l:RemoveVertex}
Suppose $V$ and $W$ are $\TT^1$-disjoint edge-compressing disks, such that every other disk that is $\TT^1$-disjoint from $V$ is also disjoint from $W$. Then deleting $[V]$ from $\plex{\EE(H)}$ produces a complex with the same homotopy type as $\plex{\EE(H)}$.
\end{lem}

\begin{proof}
The hypothesis of the lemma guarantees that for every simplex $\sigma$ of $\plex{\EE(H)}$ that has $[V]$ as a vertex, there is a simplex $\tau$ with $[W]$ as a vertex such that $\tau=[W] \cup \sigma$. Contracting the edge $e$ that joins $[V]$ and $[W]$ then collapses $\tau$ onto the simplex $\tau -[V]$. Hence, this contraction is a strong deformation retract onto the subcomplex of $\plex{\EE(H)}$ obtained by deleting $[V]$. 
\end{proof}

As noted above, if $H$ meets $\TT^1$ $4k$ times, then there are $2k-3$ lines in each family $\MV$ and $\MW$. Furthermore, there are three points of $H \cap \TT^1$ on each side of each line family. The middle point of each of these triples will be called a {\it central point} of that line family. We number the points of $H \cap \TT^1$, so that $1$ is a central point of $\MV$, and one of the central points of $\MW$ is numbered $p \le k$, as in Figure \ref{f:12-gonCircle}(b). We will call the number $p-1$ the {\it offset} of $H$. For example, the offset of the 12-gon $H$ depicted in Figure \ref{f:12-gonCircle}(a) is 2. Figure \ref{f:20-gonCircles} depicts the line families $\MV$ and $\MW$ for the two 20-gons of Figure \ref{f:20-gons}, showing that their offsets are 2 and 4. 

\begin{lem}
\label{l:helicoid}
If the offset of $H$ is 2 then $H$ is a normal helicoid.
\end{lem}

\begin{proof}
When the offset is 2 there are two points $p$ and $p'$ of $H \cap \TT^1$ that do not meet any edge-compressing disk. As in the proof of Lemma \ref{l:NormalTriangle}, it follows that these points lie on edges of $\TT^1$ that meet $\bdy H$ once. (See, for example, the 12-gon of Figure \ref{f:12-gonCircle} and the 20-gon of Figure \ref{f:20-gons} that has offset 2.) It is well known that a normal loop meets opposite edges of a tetrahedron in the same number of points, so the two edges $e$ and $e'$ that contain $p$ and $p'$ must be opposite. The surface $\bdy \Delta-(e \cup e')$ is thus an annulus $A$, with the other edges of $\TT^1$ connecting opposite ends of $A$. Since $\bdy H$ is made up of normal arcs, it follows that the two arcs $\bdy H-(p \cup p')$ must spiral monotonically around $A$. This is precisely the intersection pattern of a helicoid $\Sigma$ with the boundary of a tetrahedron $\Delta$, when opposite edges of $\TT^1$ meet the central axis of $\Sigma$. (See Figure \ref{f:helicoid}.) Such an intersection defines a normal helicoid in $\Delta$. 
\end{proof}

\begin{figure}
\psfrag{1}{\tiny $1$}
\psfrag{2}{\tiny $2$}
\psfrag{3}{\tiny $3$}
\psfrag{4}{\tiny $4$}
\psfrag{5}{\tiny $5$}
\psfrag{6}{\tiny $6$}
\psfrag{7}{\tiny $7$}
\psfrag{8}{\tiny $8$}
\psfrag{9}{\tiny $9$}
\psfrag{a}{\tiny $10$}
\psfrag{b}{\tiny $11$}
\psfrag{c}{\tiny $12$}
\psfrag{d}{\tiny $13$}
\psfrag{e}{\tiny $14$}
\psfrag{f}{\tiny $15$}
\psfrag{g}{\tiny $16$}
\psfrag{h}{\tiny $17$}
\psfrag{i}{\tiny $18$}
\psfrag{j}{\tiny $19$}
\psfrag{k}{\tiny $20$}
\psfrag{x}{(a)}
\psfrag{y}{(b)}
\[\includegraphics[width=4.5in]{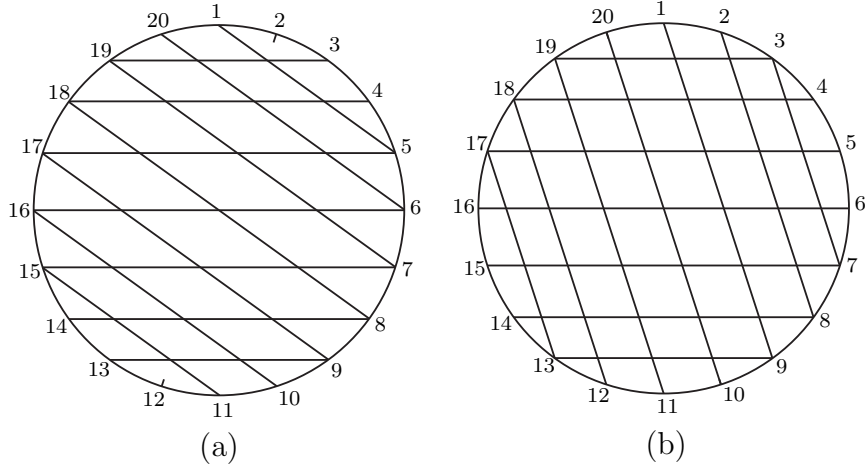}\]
\caption{The line families $\MV$ and $\MW$ associated with the 20-gons depicted in Figure \ref{f:20-gons}. In (a) the offset is 2, and in (b) the offset is 4.}
\label{f:20-gonCircles}
\end{figure}

\begin{lem}
There are no disks in $\Delta$ with offset 0 or 1.
\end{lem}

\begin{proof}
If $H$ has offset 0 or 1, then the two hemispheres of $\bdy \Delta$ meet as in Figure \ref{f:offset0and1}(a) or Figure \ref{f:offset0and1}(b). In either case, $\TT^1$ bounds a bi-gon on $\bdy \Delta$, a contradiction. 
\end{proof}

\begin{figure}
\[\begin{array}{cc}
\includegraphics[width=2 in]{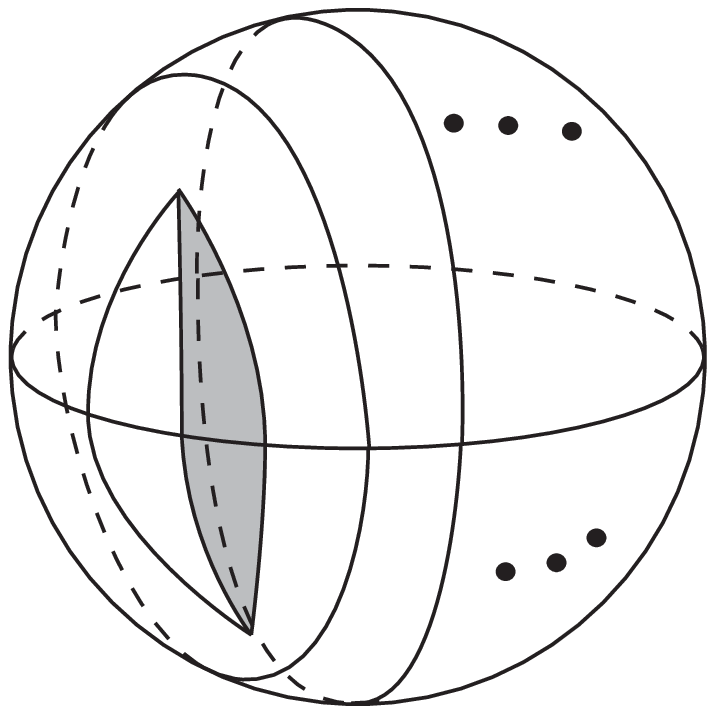} & \includegraphics[width=2 in]{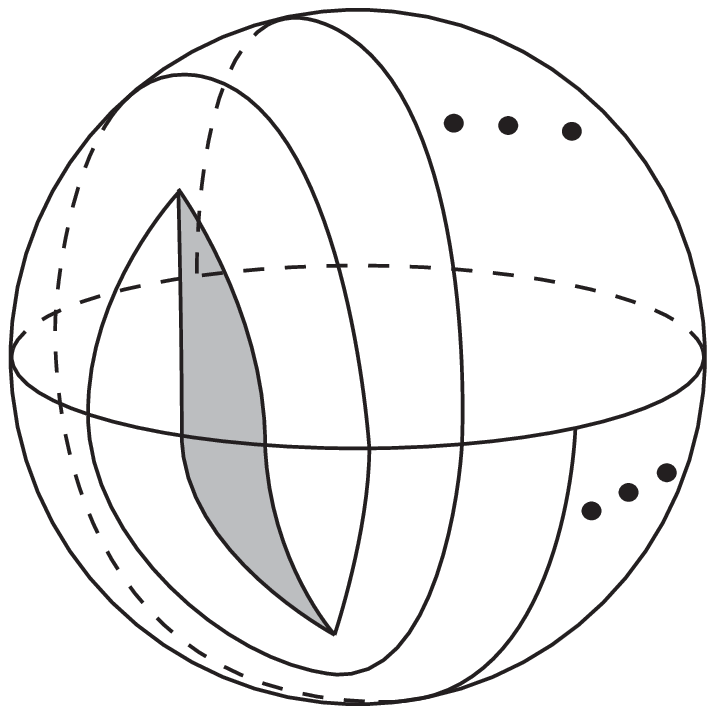}\\
\mbox{(a)} & \mbox{(b)}
\end{array}\]
\caption{The sphere $\bdy \Delta$ when $H$ has (a) offset 0 and (b) offset 1.}
\label{f:offset0and1}
\end{figure}

\begin{lem}
If $H$ is a disk in $\Delta$ with normal boundary and offset 2, and $|H \cap \TT^1|=4(n+1)$, then $H$ has topological index $n$. 
\end{lem}

\begin{proof}
There are $2n-1$ lines in $\MV$ and $\MW$. Number these lines as in Figure \ref{f:offset2}(a). We first show that the even numbered lines of $\MV$ and $\MW$ can be removed, and the homotopy type of the subcomplex defined by the disks represented by the remaining lines will be the same as that of the original complex $\plex{\EE(H)}$. 

First, observe that every line in both $\MV$ and $\MW$ that is $\TT^1$-disjoint from line 2 of $\MV$ is also $\TT^1$-disjoint from line 1 of $\MV$. Hence, by lemma \ref{l:RemoveVertex} we can remove line 2 of $\MV$, and the homotopy type of the complex defined by the remaining disks is unchanged. By a symmetric argument we remove line 2 of $\MW$, and line $2n-2$ of both $\MV$ and $\MW$. 

\begin{figure}
\psfrag{1}{\tiny $1$}
\psfrag{2}{\tiny $2$}
\psfrag{3}{\tiny $3$}
\psfrag{4}{\tiny $4$}
\psfrag{5}{\tiny $5$}
\psfrag{6}{\tiny $6$}
\psfrag{7}{\tiny $7$}
\psfrag{x}{(a)}
\psfrag{y}{(b)}
\[\includegraphics[width=4.5in]{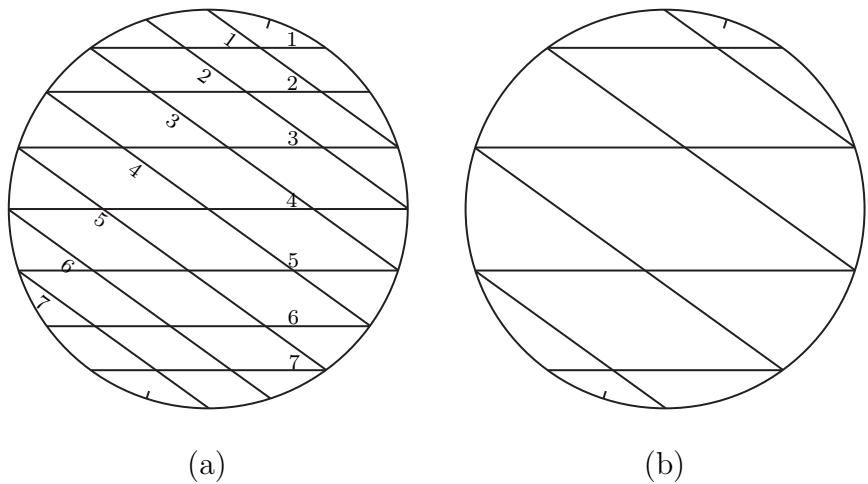}\]
\caption{(a) Numbering the line families $\MV$ and $\MW$ when the offset is 2. (b) The line families after the removal of the even numbered lines.}
\label{f:offset2}
\end{figure}

We now proceed by induction. Suppose $i$ is an even number less than $\frac{2n-1}{2}$, and we have removed all lines in $\MV$ and $\MW$ with even numbers less than $i$. Then every line of $\MV$ and $\MW$ that is $\TT^1$-disjoint from line $i$ of $\MV$ is also $\TT^1$-disjoint from line $i-1$ of $\MV$. Hence, line $i$ of $\MV$ can be removed. By a symmetric argument, we remove line $i$ of $\MW$, and line $2n-1-i$ of $\MV$ and $\MW$. 

Once the even numbered lines in both families have been removed, there are $n$ lines remaining in each family. See Figure \ref{f:offset2}(b). Furthermore, line $i$ in $\MV$ and line $j$ in $\MW$ are $\TT^1$-disjoint if and only if $i \ne j$. It follows that the complex determined by these lines is homotopy equivalent to the join of $n$ 0-spheres. This is homeomorphic to $S^{n-1}$, and thus the complex $\plex{\EE(H)}$ has index $n$.
\end{proof}

\begin{figure}
\psfrag{a}{$\alpha$}
\psfrag{b}{$\beta$}
\[\includegraphics[width=2.5 in]{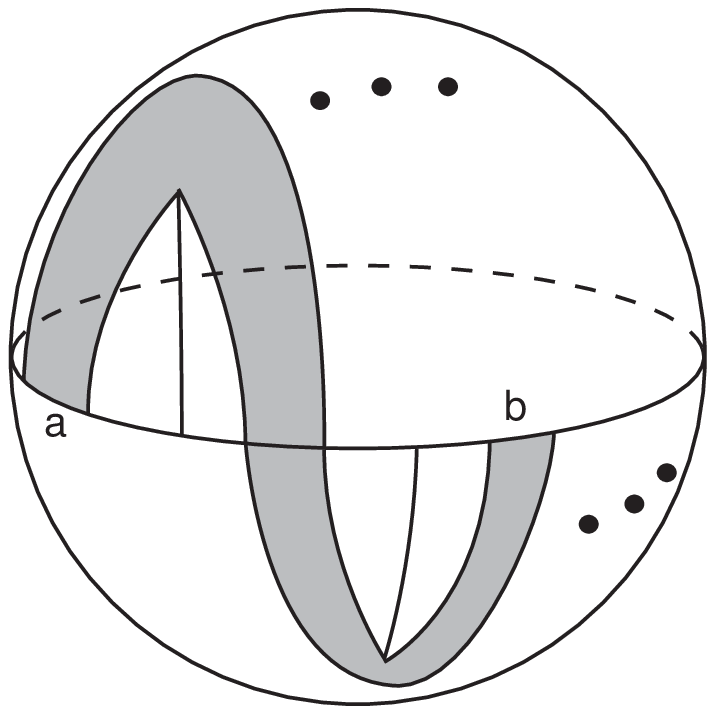}\]
\caption{The sphere $\bdy \Delta$ with offset 3.}
\label{f:offset3}
\end{figure}

\begin{lem}
There are no disks in $\Delta$ with offset 3.
\end{lem}

\begin{proof}
When the offset is 3, the sphere $\bdy \Delta$ is depicted in Figure \ref{f:offset3}. The shaded region in the figure is a subset of a 2-simplex $\sigma$, bounded by the arcs $\alpha$ and $\beta$ of $\bdy H$. This region meets two vertices of $\sigma$. Since there is only one more vertex of $\sigma$, it can not be the case that both $\alpha$ and $\beta$ are normal arcs. 
\end{proof}

\begin{figure}
\psfrag{1}{\tiny $1$}
\psfrag{2}{\tiny $2$}
\psfrag{3}{\tiny $3$}
\psfrag{4}{\tiny $4$}
\psfrag{5}{\tiny $5$}
\psfrag{6}{\tiny $6$}
\psfrag{7}{\tiny $7$}
\psfrag{8}{\tiny $8$}
\psfrag{9}{\tiny $9$}
\psfrag{a}{\tiny $10$}
\psfrag{b}{\tiny $11$}
\psfrag{c}{\tiny $12$}
\psfrag{d}{\tiny $13$}
\psfrag{e}{\tiny $14$}
\psfrag{f}{\tiny $15$}
\psfrag{g}{\tiny $16$}
\psfrag{h}{\tiny $17$}
\psfrag{i}{\tiny $18$}
\psfrag{j}{\tiny $19$}
\psfrag{k}{\tiny $20$}
\psfrag{x}{(a)}
\psfrag{y}{(b)}
\psfrag{w}{$w$}
\psfrag{v}{$v$}
\[\includegraphics[width=4.5in]{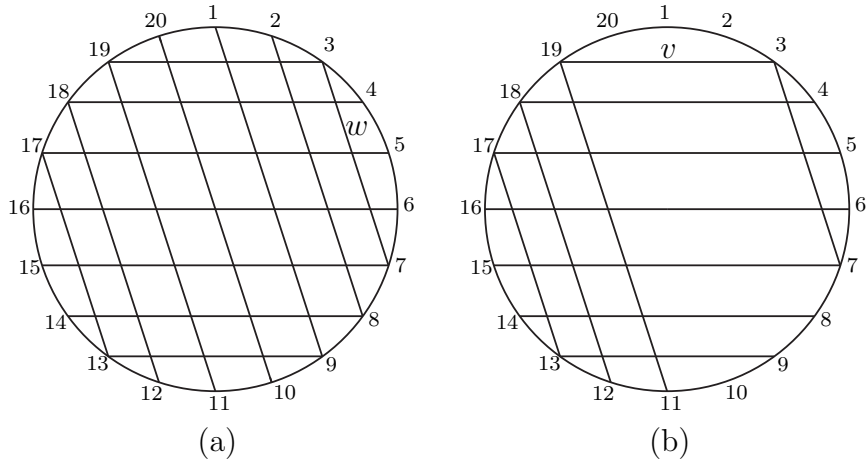}\]
\caption{(a) Any line in $\MV$ or $\MW$ that is $\TT^1$-disjoint from the lines of $\MW$ that meet the points 1, 2, or 20 is also $\TT^1$-disjoint from $w$. (b) After removal, the line $v$ is $\TT^1$-disjoint from every line that remains.}
\label{f:offset4}
\end{figure}

\begin{lem}
\label{l:offset4}
If $H$ is a disk in $\Delta$ with normal boundary and offset at least 4, then $H$ is not topologically minimal. 
\end{lem}

\begin{proof}
Let $w$ denote the line of $\MW$ that meets point 3 (since the offset is at least 4, such a line must exist). Notice that the line of $\MW$ that meets the point 1, and the two lines of $\MW$ on either side of it, have the property that anything in $\MV$ or $\MW$ that is $\TT^1$-disjoint from them is also $\TT^1$-disjoint from $w$. See Figure \ref{f:offset4}(a). Hence, by Lemma \ref{l:RemoveVertex}, these three lines can be removed. However, once this removal takes place, the line $v$ of $\MV$ that is closest to point 1 is now $\TT^1$-disjoint from every other line in both $\MV$ and $\MW$. See Figure \ref{f:offset4}(b). We conclude that the subcomplex of  $\plex{\EE(H)}$ defined by the remaining lines is a cone on the vertex corresponding to $v$. Hence, this subcomplex is contractible, a contradiction. 
\end{proof}

Theorem \ref{t:main} immediately follows from Lemma \ref{l:NormalTriangle} and Lemmas \ref{l:helicoid} through \ref{l:offset4}.

\bibliographystyle{alpha}

\end{document}